\documentclass[preprint,10pt]{elsarticle}

\usepackage{mathrsfs}
\usepackage{amssymb} 
\usepackage{amsmath}
\usepackage{hyperref}
\usepackage[linesnumbered, ruled]{algorithm2e}
\usepackage{multirow}
\usepackage{colortbl}

\usepackage{graphicx}  
\usepackage{float}  
\usepackage{subfigure}  

\usepackage{epstopdf}

\usepackage{booktabs}

\usepackage{algorithmic}
\usepackage{multirow}
\usepackage{amsmath}
\usepackage{xcolor}

\newtheorem{theorem}{Theorem}

\newtheorem{lemma}[theorem]{Lemma}
\newtheorem{definition}[theorem]{Definition}

\newtheorem{example}[theorem]{Example}
\newenvironment{proof}{{\bf Proof.}}{\hspace{0.35cm} $\Box$}

\journal{}

\begin{document}

\begin{frontmatter}

\title{Computing Gr\"{o}bner bases of ideal interpolation\tnoteref{label1}}
\tnotetext[label1]{This work was supported  by National Natural Science Foundation of China under Grant No. 11901402.}

\author{Xue Jiang\corref{2} }
\ead{littledonna@163.com}

\author{Yihe Gong\corref{cor1}}
\ead{yhegong@163.com}

\address{Xue Jiang: School of Mathematics and Statistics, Changchun University of Science and Technology, Changchun, China;
 Yihe Gong(Corresponding author): College of Science, Northeast Electric Power University, Jilin, China
}

\cortext[cor1]{Corresponding author.}

\begin{abstract}
 We present algorithms for computing the reduced Gr\"{o}bner basis of the vanishing ideal of a finite set of points in a frame of ideal interpolation. Ideal interpolation is defined by a linear projector whose kernel is a polynomial ideal. In this paper, we translate interpolation condition functionals into formal power series via Taylor expansion, then the
  reduced Gr\"{o}bner basis is read from formal power series by Gaussian elimination. Our algorithm has a polynomial time complexity, an example suggests that our method compares favorably with MMM algorithm in some special cases.
\end{abstract}
\end{frontmatter}

\section{Introduction}\label{sec:Int}
Let $\mathbb{F}$ be either the real field $\mathbb{R}$ or the complex field $\mathbb{C}$. Polynomial interpolation is to construct a polynomial $g$ belonging to a finite-dimensional subspace of $\mathbb{F}[{\bf X}]$ that agrees with a given function $f$ at the data set, where $\mathbb{F}[{\bf X}] := \mathbb{F}[x_1, x_2, \dots, x_d]$ denotes the polynomial ring in $d$ variables over the field $\mathbb{F}$.

For studying multivariate polynomial interpolation, Birkhoff \cite{Birkhoff} introduced the definition of ideal interpolation. Ideal interpolation is defined by a linear projector whose kernel is a polynomial ideal. In ideal interpolation \cite{de boor2005}, the interpolation condition functionals at an interpolation point $\boldsymbol{\theta} \in \mathbb{F}^d$ can be described by a linear space ${\rm span}\{\delta_{\boldsymbol{\theta}} \circ P(D), P \in P_{\boldsymbol{\theta}}\}$, where $P_{\boldsymbol{\theta}}$ is a $D$-invariant polynomial subspace, $\delta_{\boldsymbol{\theta}}$ is the evaluation functional at $\boldsymbol{\theta}$ and $P(D)$ is the differential operator induced by $P$. The classical examples of ideal interpolation are Lagrange interpolation and univariate Hermite interpolation.

For an ideal interpolation, suppose that $\Delta$ is the finite set of interpolation condition functionals. Then the set of all polynomials that vanish at $\Delta$ constitutes a 0-dimensional ideal, which is denoted by $I(\Delta)$, namely,
$$I(\Delta):=\{f \in \mathbb{F}[{\bf X}]: L(f)=0, \forall L \in \Delta\}.$$
We refer the readers to \cite{de boor2005} and \cite{Boris09new} for more details about ideal interpolation.

Gr\"{o}bner bases \cite{Cox2007}, introduced by Buchberger \cite{Buchberger} in 1965, have been applied successfully in various fields of mathematics and to many types of problems. There are several methods for computing Gr\"{o}bner bases of vanishing ideals. For any point set $\Theta \subset \mathbb{F}^d$ and any fixed monomial order, BM algorithm \cite{BM1982} yields the reduced Gr\"{o}bner basis and a reducing interpolation Newton basis for a $d$-variate Lagrange interpolation on $\Theta$. BM algorithm computes the vanishing ideal of a finite set of points (in affine space) without multiplicities.
\cite{ABKR} presents a variant of BM algorithm  which is more effective for the computation over $\mathbb{Q}$, and also considers the case that the set of points is in the projective space. As a generalization of univariate Newton interpolation, Farr and Gao \cite{FarrGao} give an algorithm that computes the reduced Gr\"{o}bner basis for vanishing ideals under any monomial order. Applying Taylor expansions to the original algorithm, Farr and Gao's method can be applied to compute the vanishing ideal when the interpolation points have multiplicities, but the multiplicity set needs to be a delta set in $\mathbb{N}^d$. In this paper, we consider the more general case that the multiplicity space of each point just needs to be closed under differentiation.  To avoid solving linear equations, Lederer \cite{Lederer} gives an algorithm to compute the Gr\"{o}bner basis of an arbitrary finite set of points under lexicographic order by induction over the dimension $d$. Jiang, Zhang and Shang \cite{Xue Jiang} give an algorithm for computing the Gr\"{o}bner basis of a single point ideal interpolation.

To the best of our knowledge, for the general case of ideal interpolation, computing the reduced Gr\"{o}bner basis needs to solve a linear system. In 1993, Marinari, M\"{o}ller and Mora \cite{MMM1993} constructed a linear system based on a Vandermonde-like matrix,  and gave an algorithm (MMM algorithm) to compute general 0-dimensional ideals by Gaussian elimination. MMM algorithm has a polynomial time complexity and is one of the most famous algorithms in recent years.


Throughout the paper, $\mathbb{N}$ denotes the set of nonnegative integers. Let $\mathbb{N}^d := \{ (\alpha_1, \alpha_2, \dots, \alpha_d) :  \alpha_i \in \mathbb{N}\}$.  For  $\boldsymbol{\alpha}=(\alpha_1, \alpha_2, \dots, \alpha_d) \in \mathbb{N}^d$,
define $\boldsymbol{\alpha}!:=\alpha_1!\alpha_2!\cdots\alpha_d!$
and denote by ${\bf X}^{\boldsymbol{\alpha}}$ the monomial $x_1^{\alpha_1} x_2^{\alpha_2} \cdots x_d^{\alpha_d}$. Let $\{{\bf X}^{\boldsymbol{\alpha_1}},{\bf X}^{\boldsymbol{\alpha_2}},\dots, {\bf X}^{\boldsymbol{\alpha_n}}\}$ be a finite set of monomials and $\{L_i: \mathbb{F}[{\bf X}] \to \mathbb{F}, i=1,2,\dots, n\}$ be a finite set of linearly independent functionals. We can treat the matrix
\[\left( {\begin{array}{*{20}{c}}
{{L_1}({{\bf X}^{\boldsymbol{\alpha_1}}})}&{{L_1}({{\bf X}^{\boldsymbol{\alpha_2}}})}& \cdots &{{L_1}({{\bf X}^{\boldsymbol{\alpha_n}}})}\\
{{L_2}({{\bf X}^{\boldsymbol{\alpha_1}}})}&{{L_2}({{\bf X}^{\boldsymbol{\alpha_2}}})}& \cdots &{{L_2}({{\bf X}^{\boldsymbol{\alpha_n}}})}\\
 \vdots & \vdots & & \vdots \\
{{L_n}({{\bf X}^{\boldsymbol{\alpha_1}}})}&{{L_n}({{\bf X}^{\boldsymbol{\alpha_2}}})}& \cdots &{{L_n}({{\bf X}^{\boldsymbol{\alpha_n}}})}
\end{array}} \right)\]
as a Vandermonde-like matrix.

%

In this paper, we give algorithms to compute the reduced Gr\"{o}bner basis by Gaussian elimination. Our major idea is based on the formal power series appeared in \cite{de Boor90}.
In that paper, de Boor and Ron deduced excellent properties for the basis of interpolation they computed as the least degree forms of formal power series.
In this paper, we focus on extracting the reduced Gr\"{o}bner basis from interpolation condition functionals. We translate interpolation condition functionals into formal power series via Taylor expansion, then the reduced Gr\"{o}bner basis is read from formal power series by Gaussian elimination.

 The paper is organized as follows. The preliminary is in Section \ref{sec:Preliminaries}. The algorithm for computing the ``reverse" reduced team is in Section \ref{sec:3}.  The method to find the quotient ring basis is in Section \ref{sec:4}. The algorithms for computing the reduced Gr\"{o}bner bases are presented in Section \ref{sec:5}. A special example  is discussed in Section \ref{sec:6}.

\section{Preliminary}\label{sec:Preliminaries}

 A polynomial $P \in \mathbb{F}[{\bf X}]$ can be considered as the formal power series
\begin{displaymath}
    P=\sum_{\boldsymbol{\alpha} \in \mathbb{N}^d} \hat{P}(\boldsymbol{\alpha}) {\bf X}^{\boldsymbol{\alpha}},
\end{displaymath}
where $\hat{P}(\boldsymbol{\alpha})$ are the coefficients in the polynomial $P$.

$P(D):=P(D_{x_1}, D_{x_2}, \dots, D_{x_d})$ is the differential operator induced by the polynomial $P$, where $D_{x_j}:=\frac{\partial}{\partial x_j}$ is the differentiation with respect to the $j$th variable, $j=1, 2, \dots, d$.

Define $D^{\boldsymbol{\alpha}}:=D_{x_1}^{\alpha_1}D_{x_2}^{\alpha_2}\cdots D_{x_d}^{\alpha_d}$. The differential polynomial can be rewritten as
\begin{displaymath}
    P(D)=\sum_{\boldsymbol{\alpha}\in \mathbb{N}^d} \hat{P}(\boldsymbol{\alpha})D^{\boldsymbol{\alpha}}.
\end{displaymath}

Given a monomial order $\prec$, the least monomial of the polynomial $P$ w.r.t. $\prec$ is defined by
\begin{displaymath}
    {\rm lm} (P) := \min_{\prec}  \{ {\bf X}^{\boldsymbol{\alpha}} \mid \hat{P}(\boldsymbol{\alpha}) \neq 0\}.
\end{displaymath}

	\begin{definition} \rm
	    We denote by $\Lambda\{P_1, P_2, \dots, P_n\}$ the set of all monomials that occur in  the polynomials $P_1, P_2, \dots, P_n$ with nonzero coefficients.
	\end{definition}
	
For example, let $P_1=1, P_2=x, P_3=\frac{1}{2}x^2 +y$, then
	
\begin{displaymath}
	        \Lambda \{P_1, P_2, P_3\} = \{1, x, y, x^2\}.
	    \end{displaymath}

\begin{definition}\rm
	    Given a monomial order $\prec$, a set of linearly independent polynomials $\{P_1, P_2, \dots, P_n\} \subset \mathbb{F}[{\bf X}]$ is called a ``reverse" reduced team w.r.t.$\prec$, if
		\begin{enumerate}
		 \item the coefficient of the least monomial of the polynomial $P_i, 1 \le i \le n$ is 1;
		 \item ${\rm lm}(P_i) \not\in \Lambda \{P_j\}, i\neq j,  1 \le i , j \le n$.
		\end{enumerate}
\end{definition}

For example, given the monomial order $\mathrm{grlex}(z \prec y \prec x)$,
			    $$\{1, x, \frac{1}{2}x^2 + y, \frac{1}{6} x^3 - x^2 + xy\}, ~~
			    \{1, y+z, x\},~~ \{1, x+z, -x+y\}$$
		are ``reverse" reduced teams  w.r.t. $\prec$.  	

\section{ Computing a ``reverse" reduced team by Gaussian elimination} \label{sec:3}

Let $T$ and $\bar{T}$ be two sets of monomials in $\mathbb{F}[{\bf X}]$, the notation $\bar{T}-T$ is reserved for the set $\{t : t\in \bar{T}, t\notin T\}$.

Given a monomial order $\prec$, $P_1, P_2, \dots, P_n \in \mathbb{F}[{\bf X}]$ are linearly independent polynomials. Algorithm \ref{alg:revredbasis} yields a ``reverse" reduced team  w.r.t. $\prec$.
	
	\begin{algorithm}
	\caption{A ``reverse" reduced team  w.r.t. $\prec$}
	\label{alg:revredbasis}
	\begin{algorithmic}[1]
	\STATE{{\bf Input:} A monomial order $\prec$.}
	\STATE{\hspace*{1.1cm}~Linearly independent polynomials $P_1, P_2, \dots, P_n \in \mathbb{F}[{\bf X}]$.}
	\STATE{{\bf Output:} $\{P_1^{*}, P_2^{*}, \dots, P_n^{*}\}$, a ``reverse" reduced team w.r.t. $\prec$.}	
	\STATE{//Initialization}
	\STATE{${\rm{List}}:=\Lambda\{P_1, P_2, \dots, P_n\}$};
    \STATE{${\bf X}^{\boldsymbol{\alpha}}:= \rm{min} (\rm{List},\prec)$};

	\STATE{${\rm{U}}:=(\hat{P}_1(\boldsymbol{\alpha}),\hat{P}_2(\boldsymbol{\alpha}),\dots, \hat{P}_n(\boldsymbol{\alpha}))'$;}
    \STATE{$\rm{List}:=\rm{List}-\{{\bf X}^{\boldsymbol{\alpha}}\}$};
    \WHILE{$\rm{List}\neq \emptyset$ }
    \STATE{${\bf X}^{\boldsymbol{\alpha}}:= \rm{min} (\rm{List},\prec)$};
    \STATE{$ v:=(\hat{P}_1(\boldsymbol{\alpha}),\hat{P}_2(\boldsymbol{\alpha}),\dots, \hat{P}_n(\boldsymbol{\alpha}))'$;}
    \STATE{${\rm{U}}:=[{\rm{U}},v]$};
    \STATE{$\rm{List}:=\rm{List}-\{{\bf X}^{\boldsymbol{\alpha}}\}$};

   \ENDWHILE

   \STATE{//Computing}
   \STATE{$\rm{U}^{*}=\rm{rref}[\rm{U}]$}; (reduced row echelon form)
  \STATE{${\rm{List}}:=\Lambda\{P_1, P_2, \dots, P_n\}$};
\FOR{$i=1:n$}
			\STATE{$P_i^{*}={\rm{U}}^{*} (i,:)\cdot (\rm{List},\prec)'$;}
	\ENDFOR
	\RETURN $\{P_1^{*}, P_2^{*}, \dots, P_n^{*}\}$.
	\end{algorithmic}
	\end{algorithm}	
For example, given the monomial order $\mathrm{grlex}(y \prec x)$, $P_1=1$, $P_2=x$,$P_3=\frac{1}{2}x^2+y$, $P_4=\frac{1}{6}x^3+xy+2y$. We get
$$({\rm{List}},\prec)=(1,y,x,xy,x^2,x^3),$$
\[{\rm{U}} = \left( {\begin{array}{*{20}{c}}
1&0&0&0&0&0\\
0&0&1&0&0&0\\
0&1&0&0&{1/2}&0\\
0&2&0&1&0&{1/6}
\end{array}} \right),\]
\[{{\rm{U}}^*} = \left( {\begin{array}{*{20}{c}}
1&0&0&0&0&0\\
0&1&0&0&{1/2}&0\\
0&0&1&0&0&0\\
0&0&0&1&-1&{1/6}
\end{array}} \right),\]	
${P_1}^*=1$, ~~${P_2}^*=\frac{1}{2}x^2+y$,~~ ${P_3}^*=x$, ~~${P_4}^*=\frac{1}{6}x^3-x^2+xy$.
	
 \section{Interpolation monomial basis and quotient ring basis} \label{sec:4}

 Given interpolation conditions $\Delta ={\rm span}\{L_1, L_2, \dots, L_n\}$, where $L_1, L_2, \dots, L_n$ are linearly independent functionals.  Let $T=\{{\bf X}^{\boldsymbol{\alpha_1}},{\bf X}^{\boldsymbol{\alpha_2}},\dots, {\bf X}^{\boldsymbol{\alpha_n}}\}$ be a set of monomials, then $T$ is an interpolation monomial basis for $\Delta$ if the Vandermonde-like matrix ( applying the data map $\{L_i: i=1,2, \dots, n\}$ to the column map $\{{\bf X}^{\boldsymbol{\alpha_j}}: j=1,2, \dots,n\}$) is non-singular.

\begin{definition}
	Let $T$ and $\bar{T}$ be two sets of monomials in $\mathbb{F}[{\bf X}]$ with $\bar{T}-T\neq \emptyset$ and $T-\bar{T}\neq \emptyset$. Given a monomial order $\prec$, we write $\bar{T}\prec T$, if
	\begin{displaymath} \abovedisplayskip=2pt \belowdisplayskip=2pt
	    \max_{\prec} (\bar{T}-T)\prec \max_{\prec} (T-\bar{T}).
	\end{displaymath}
	\end{definition}

\begin{definition}[$\prec$-minimal monomial basis \cite{SauerT1998}]
	    Given a monomial order $\prec$ and interpolation conditions $\Delta ={\rm span}\{L_1, L_2, \dots, L_n\}$, where $L_1, L_2, \dots, L_n$ are linearly independent functionals. Let $T$ be an interpolation monomial basis for $\Delta$, then $T$ is $\prec$-minimal if there exists no interpolation monomial basis $\bar{T}$ for $\Delta$ satisfying $\bar{T} \prec T$.
	\end{definition}

First, we consider the interpolation problem at the origin.

	\begin{lemma} \label{lem:InterpCondHasMon}
	    Given interpolation conditions $\Delta = \delta_{\bf 0} \circ {\rm span}\{P_1(D), P_2(D), \dots, P_n(D)\}$, where $P_1, P_2, \dots, P_n \in \mathbb{F}[{\bf X}]$ are linearly independent polynomials. Let $T=\{{\bf X}^{\boldsymbol{\beta_1}}, {\bf X}^{\boldsymbol{\beta_2}}, \dots, {\bf X}^{\boldsymbol{\beta_n}}\}$ be an interpolation monomial basis for $\Delta$, then for each $ P_i, 1 \le i \le n$, there exists an ${\bf X}^{\boldsymbol{\alpha_i}} \in \Lambda\{P_i\}$ satisfying ${\bf X}^{\boldsymbol{\alpha_i}} \in T$.
	\end{lemma}

	\begin{proof}
	    We will prove this by contradiction. Without loss of generality, we can assume that for every ${\bf X}^{\boldsymbol{\alpha}} \in \Lambda\{P_1\}$, ${\bf X}^{\boldsymbol{\alpha}} \not\in T$. It is observed that
		\begin{displaymath}
		    \begin{aligned} \abovedisplayskip=2pt \belowdisplayskip=2pt
		        [\delta_{\bf 0} \circ P_1(D)] {\bf X}^{\boldsymbol{\beta_j}} & = [\delta_{\bf 0} \circ \sum \hat{P_1}(\boldsymbol{\alpha}) D^{\boldsymbol{\alpha}}] {\bf X}^{\boldsymbol{\beta_j}} \\
		           &= \sum \hat{P_1}(\boldsymbol{\alpha}) \underbrace{(\delta_{\bf 0} \circ D^{\boldsymbol{\alpha}} {\bf X}^{\boldsymbol{\beta_j}})}_{0} \\
				   &= 0, \quad 1 \le j \le n.
		    \end{aligned}
		\end{displaymath}
So the Vandermonde-like matrix has a zero row, and it is singular. It contradicts with the condition that $T=\{{\bf X}^{\boldsymbol{\beta_1}}, {\bf X}^{\boldsymbol{\beta_2}}, \dots, {\bf X}^{\boldsymbol{\beta_n}}\}$ is an interpolation monomial basis for $\Delta$.
	\end{proof}

	    Given interpolation conditions $\Delta = \delta_{\bf 0} \circ {\rm span}\{P_1(D), P_2(D), \dots, P_n(D)\}$, Lemma \ref{lem:InterpCondHasMon} shows that we need to choose at least one monomial from each  $P_i, 1 \le i \le n$ to form the interpolation monomial basis for $\Delta$.

	 For $\boldsymbol{\theta}=(\theta_1, \theta_2, \dots, \theta_d) \in \mathbb{F}^d$, we denote $\boldsymbol{\theta} {\bf X} :=\sum_{i=1}^d \theta_i x_i$. From Taylor's formula, we have
	\begin{displaymath}
		e^{\boldsymbol{\theta} {\bf X}} = \sum_{j=0}^\infty \frac{(\boldsymbol{\theta} {\bf X})^j}{j!}.
	\end{displaymath}
Then, we have

\begin{equation} \label{eqn:nonzeroCon}
	    \delta_{\boldsymbol{\theta}}=
		\delta_{\bf 0} \circ  e^{\boldsymbol{\theta} D}.
\end{equation}
Further more, we get
	\begin{equation} \label{eqn:nonzeroCons2zeroCons}
	    \delta_{\boldsymbol{\theta}} \circ P(D) =
		\delta_{\bf 0} \circ  e^{\boldsymbol{\theta} D} P(D).
	\end{equation}
	This means that an interpolation problem at a nonzero point can be converted into one at the origin.

\begin{theorem}\label{thm:6}
	    For any $\boldsymbol{\theta} \in \mathbb{F}^d$, given a monomial order $\prec$ and interpolation conditions $\Delta =\delta_{\boldsymbol{\theta}} \circ {\rm span}\{P_1(D), P_2(D), \dots, P_n(D)\}$. If $\{P_1, P_2, \dots, P_n\}$ is a ``reverse" reduced team w.r.t. $\prec$, then $\{{\rm lm}(P_1), {\rm lm}(P_2), \dots, {\rm lm}(P_n)\}$ is the $\prec$-minimal monomial basis for $\Delta$.
	\end{theorem}

	\begin{proof}
     Let ${\tilde {P}_i}(D) = {e^{\boldsymbol{\theta} D}} P_i(D), 1 \le i \le n$.	By (\ref{eqn:nonzeroCons2zeroCons}), we have
     $$\Delta =\delta_{\boldsymbol{\theta}} \circ {\rm span}\{P_1(D), P_2(D), \dots, P_n(D)\}=\delta_{\bf 0} \circ {\rm span}\{\tilde {P}_1(D), \tilde {P}_2(D), \dots, \tilde {P}_n(D)\}.$$
According to Lemma \ref{lem:InterpCondHasMon}, we choose at least one monomial from each $\tilde {P}_i, 1 \le i \le n$ to form the interpolation monomial basis. On the other hand, $P_i, 1 \le i \le n$ are assumed to be linearly independent,  which implies that the cardinal number of interpolation monomial basis is $n$. Thus, it is easy to see $\{{\rm lm}(\tilde {P}_1), {\rm lm}(\tilde {P}_2), \dots, {\rm lm}(\tilde {P}_n)\}$ is minimal choice w.r.t. $\prec$. We only need to prove $\{{\rm lm}(\tilde {P}_1), {\rm lm}(\tilde {P}_2), \dots, {\rm lm}(\tilde {P}_n)\}$ is an interpolation monomial basis for $\Delta$.

		Let $P_i =\sum \hat{P_i}(\boldsymbol{\alpha}) {\bf X}^{\boldsymbol{\alpha}} +  {\bf X}^{\boldsymbol{\beta_i}}$, ${\rm lm}(P_i) = {\bf X}^{\boldsymbol{\beta_i}}, 1 \le i \le n$. Since $\{P_1, P_2, \dots, P_n\}$ is a ``reverse" reduced team w.r.t. $\prec$, it means
		\begin{displaymath}
		    {\rm lm}(P_i) \not\in \Lambda \{P_j\}, \quad i\neq j, 1 \le i , j \le n.
		\end{displaymath}	
Without loss of generality, we can assume that ${\rm lm}(P_1)\prec {\rm lm}(P_2)\prec \dots \prec{\rm lm}(P_n)$, then we have
\begin{displaymath}
		    {\rm lm}(P_i) \not\in \Lambda \{\tilde {P}_j\}, \quad  1 \le i < j \le n.
		\end{displaymath}
Notice that ${\rm lm}(\tilde {P}_i)={\rm lm}( e^{\boldsymbol{\theta} {\bf X}} \cdot P_i)={\rm lm}( e^{\boldsymbol{\theta} {\bf X}}) \cdot {\rm lm}(P_i)={\rm lm}(P_i), 1 \le i \le n$, we have
\begin{displaymath}
		    {\rm lm}(\tilde {P}_i) \not\in \Lambda \{\tilde {P}_j\}, \quad  1 \le i < j \le n.
		\end{displaymath}
Thus, we have

     \begin{displaymath}
		    (\delta_{\bf 0} \circ \tilde {P}_j(D)) ({\rm lm} (\tilde {P}_i)) =
			\begin{cases}
			   0, & i< j,  \\
			    \boldsymbol{\beta_i} !  \neq 0, & i=j,  \\
			\end{cases} \quad 1 \le i, j \le n.
		\end{displaymath}
 So the Vandermonde-like matrix is an upper triangular matrix with nonzero diagonal elements, i.e., it is non-singular. It follows that
 $$\{{\rm lm}(\tilde {P}_1), {\rm lm}(\tilde {P}_2), \dots, {\rm lm}(\tilde {P}_n)\}=\{{\rm lm}(P_1), {\rm lm}(P_2), \dots, {\rm lm}(P_n)\}$$
 is the $\prec$-minimal monomial basis for $\Delta$.
	\end{proof}

In ideal interpolation,  the $\prec$-minimal monomial basis is equivalent to the monomial basis of the quotient ring w.r.t. $\prec$ \cite{SauerT1998}. The follwing theorem can be obtained directly by Theorem \ref{thm:6}, we list it  here without proof.

\begin{theorem} \label{thm:maintheorem}
	   For any $\boldsymbol{\theta} \in \mathbb{F}^d$, given a monomial order $\prec$ and interpolation conditions $\Delta =\delta_{\boldsymbol{\theta}} \circ {\rm span}\{P_1(D), P_2(D), \dots, P_n(D)\}$. If $\{P_1, P_2, \dots, P_n\}$ is a ``reverse" reduced team w.r.t. $\prec$, then $\{{\rm lm}(P_1), {\rm lm}(P_2), \dots, {\rm lm}(P_n)\}$  is the monomial basis of the quotient ring $\mathbb{F}[{\bf X}]/ I(\Delta)$ w.r.t. $\prec$.
	\end{theorem}

Now, we show the application of the Theorem \ref{thm:maintheorem}.
For example, given the monomial order $\mathrm{grlex}(y \prec x)$ and ideal interpolation conditions $\Delta =\delta_{(1, 2)} \circ {\rm span}\{1, D_x, \frac{1}{2}D_x^2+D_y, \frac{1}{6}D_x^3+D_xD_y+2D_y\}$. We know $P_1=1$, $P_2=x$,$P_3=\frac{1}{2}x^2+y$, $P_4=\frac{1}{6}x^3+xy+2y$. By Algorithm \ref{alg:revredbasis}, we have ${P_1}^*=1$, ~~${P_2}^*=\frac{1}{2}x^2+y$,~~ ${P_3}^*=x$, ~~${P_4}^*=\frac{1}{6}x^3-x^2+xy$. By Theorem \ref{thm:maintheorem}, $\{{\rm lm}(P_1^*), {\rm lm}(P_2^*), {\rm lm}(P_3^*), {\rm lm}(P_4^*)\}=\{1,y,x,xy\}$ is the monomial basis of the quotient ring $\mathbb{F}[{\bf X}]/ I(\Delta)$.


\section{The algorithms to compute the reduced Gr\"{o}bner bases} \label{sec:5}

%
%
%
%
%

In order to describe algorithms more conveniently, we introduce some notations.
Let $\mathbb{F}[[{\bf X}]]$ be the ring of formal power series.
Let $T$ be a set of monomials in $ \mathbb{F}[{\bf X}]$.
For any $f \in \mathbb{F}[[{\bf X}]]$, we denote by
\begin{displaymath}
		\lambda_T(f)= \sum\limits_{{\bf X}^{\boldsymbol{\alpha}}\in T} {\hat f(\boldsymbol{\alpha}){\bf X}^{\boldsymbol{\alpha}}}
	\end{displaymath}
a ``truncated polynomial".

Given a monomial order $\prec$ and Lagrange interpolation conditions $\Delta$, Algorithm \ref{alg:Groebnerbasis} yields the reduced Gr\"{o}bner basis for $I(\Delta)$ w.r.t. $\prec$.

\begin{algorithm}
	\caption{The reduced Gr\"{o}bner basis (Lagrange interpolation)}
	\label{alg:Groebnerbasis}
	\begin{algorithmic}[1]
	\STATE{{\bf Input:} A monomial order $\prec$.}
	\STATE{\hspace*{1.1cm}The interpolation conditions $\Delta ={\rm span}\{\delta_{\boldsymbol{\theta_1}}, \delta_{\boldsymbol{\theta_2}}, \dots, \delta_{\boldsymbol{\theta_n}}\}$,}
    \STATE{\hspace*{1.1cm}where distinct points $\boldsymbol{\theta_i} \in \mathbb{F}^d, i=1,2,\dots ,n$.}

	\STATE{{\bf Output:} $\{G_1,G_2,\dots, G_m\}$, the reduced Gr\"{o}bner basis for $I(\Delta)$ w.r.t. $\prec$.}	
	\STATE{//Initialization}
	
\STATE{$\rm{List}:=\Lambda\{e^{\boldsymbol{\theta_1} {\bf X}},e^{\boldsymbol{\theta_2} {\bf X}},\dots, e^{\boldsymbol{\theta_n} {\bf X}}\}$,~~$Q:=\{1\}$,~~$G:=\emptyset$;}
 \STATE{${\bf X}^{\boldsymbol{\beta}}:= \rm{min} (\rm{List},\prec)$};

	\STATE{$\rm{U}:=(\hat{e^{\boldsymbol{\theta_1} {\bf X}}}(\boldsymbol{\beta}),\hat{e^{\boldsymbol{\theta_2} {\bf X}}}(\boldsymbol{\beta}),\dots, \hat{e^{\boldsymbol{\theta_n} {\bf X}}}(\boldsymbol{\beta}))'$;}
    \STATE{$\rm{List}:=\rm{List}-\{{\bf X}^{\boldsymbol{\beta}}\}$};

\STATE{//Computing}
  \WHILE{$\rm{rank}(U)< n$}
 \STATE{${\bf X}^{\boldsymbol{\beta}}:= \rm{min} (\rm{List},\prec)$;}
  \STATE{$v:=(\hat{e^{\boldsymbol{\theta_1} {\bf X}}}(\boldsymbol{\beta}),\hat{e^{\boldsymbol{\theta_2} {\bf X}}}(\boldsymbol{\beta}),\dots, \hat{e^{\boldsymbol{\theta_n} {\bf X}}}(\boldsymbol{\beta}))'$;}

  \IF{${\rm{rank}}({\rm{U}},v)>{\rm{rank}}(\rm{U})$}
\STATE{${\rm{U}}:=[{\rm{U}},v]$;}
\STATE{$Q:=Q\cup \{{\bf X}^{\boldsymbol{\beta}}\}$};
 \STATE{$\rm{List}:=\rm{List}-\{{\bf X}^{\boldsymbol{\beta}}\}$};
\ELSE{}
\STATE{$G:=G\cup \{{\bf X}^{\boldsymbol{\beta}}\}$};
\STATE{$\rm{List}:=\rm{List}-\{{\rm{multiples~ of~}}{\bf X}^{\boldsymbol{\beta}}\}$};
\ENDIF
  \ENDWHILE
\STATE{${\bf X}^{\boldsymbol{\alpha}}:= \rm{min} (\rm{List},\prec)$};
\STATE{$G:=G\cup \{{\bf X}^{\boldsymbol{\alpha}}\}$};
\STATE{$G:=\{{\bf X}^{\boldsymbol{\alpha_1}}, {\bf X}^{\boldsymbol{\alpha_2}}, \dots, {\bf X}^{\boldsymbol{\alpha_m}}\}$, the set of the leading monomials of the reduced Gr\"{o}bner basis};
\STATE{$Q:=\{{\bf X}^{\boldsymbol{\beta_1}}, {\bf X}^{\boldsymbol{\beta_2}}, \dots, {\bf X}^{\boldsymbol{\beta_n}}\}$, the monomial basis of the quotient ring;}

    \STATE{$P_j:=\lambda_{G\cup Q}(e^{\boldsymbol{\theta_j} {\bf X}}),~1\leq j\leq n$;}

\STATE{$\{P_j^{*}, 1\leq j\leq n\}$, a ``reverse" reduced team  w.r.t. $\prec$, by Algorithm \ref{alg:revredbasis};}

\FOR{$i=1:m$}
            \STATE{$G_i={\bf X}^{\boldsymbol{\alpha_i}}-\sum_{j=1}^n\left(\frac{(\boldsymbol{\alpha_i})!}{(\boldsymbol{\beta_j})!}\hat{P}_j^{*}(\boldsymbol{\alpha_i})\right) {\bf X}^{\boldsymbol{\beta_j}}$;}
	\ENDFOR
	\RETURN $\{G_1,G_2,\dots, G_m\}$.
	\end{algorithmic}
	\end{algorithm}

In Line 11 and Line 14, we use the skill(recording the reversible matrix used for each calculation) in MMM algorithm to calculate the rank of the matrix by Gaussian elimination. It is obvious that Algorithm \ref{alg:Groebnerbasis} terminates. The following theorem shows its correctness.

\begin{theorem}\label{thm:lagrangeG-B}

The output $\{G_1,G_2,\dots, G_m\}$ in Algorithm \ref{alg:Groebnerbasis} is the reduced Gr\"{o}bner basis for  $I(\Delta)$.

\end{theorem}

\begin{proof} \rm
By (\ref{eqn:nonzeroCon}), $\Delta ={\rm span}\{\delta_{\boldsymbol{\theta_1}}, \delta_{\boldsymbol{\theta_2}}, \dots, \delta_{\boldsymbol{\theta_n}}\}=\delta_{\bf 0}\circ  {\rm span}\{e^{\boldsymbol{\theta_1} D}, e^{\boldsymbol{\theta_2} D}, \dots,  e^{\boldsymbol{\theta_n} D}\}$.
Suppose that $\{(e^{\boldsymbol{\theta_1} {\bf X}})^{*},(e^{\boldsymbol{\theta_2} {\bf X}})^{*},\dots, (e^{\boldsymbol{\theta_n} {\bf X}})^{*}\}$ is a ``reverse" reduced team of $e^{\boldsymbol{\theta_1} {\bf X}},e^{\boldsymbol{\theta_2} {\bf X}},\dots, e^{\boldsymbol{\theta_n} {\bf X}}$. Comparing Line 13 in Algorithm \ref{alg:Groebnerbasis} and  Line 11 in Algorithm \ref{alg:revredbasis}, we have $Q=\{{\bf X}^{\boldsymbol{\beta_1}}, {\bf X}^{\boldsymbol{\beta_2}}, \dots, {\bf X}^{\boldsymbol{\beta_n}}\}=\{{\rm lm}(e^{\boldsymbol{\theta_1} {\bf X}})^{*},{\rm lm}(e^{\boldsymbol{\theta_2} {\bf X}})^{*},\dots, {\rm lm}(e^{\boldsymbol{\theta_n} {\bf X}})^{*}\}$.
By Theorem \ref{thm:maintheorem}, $Q$ is the monomial basis of the quotient ring $\mathbb{F}[{\bf X}]/ I(\Delta)$. On the other hand, since $Q$ is a lower set \cite{Cox2007}, it is easy to check that $G$ is the set of the leading monomials of the reduced Gr\"{o}bner basis.
Thus, we only need to prove that $G_i={\bf X}^{\boldsymbol{\alpha_i}}-\sum_{j=1}^n\left(\frac{(\boldsymbol{\alpha_i})!}{(\boldsymbol{\beta_j})!}\hat{P}_j^{*}(\boldsymbol{\alpha_i})\right) {\bf X}^{\boldsymbol{\beta_j}}$ in Line 30 lies in $I(\Delta)$. Due to $\{P_1^{*}, P_2^{*}, \dots, P_n^{*}\}$ in Line 28 is a ``reverse" reduced team, it follows that ${\rm lm}(P_j^{*})={\bf X}^{\boldsymbol{\beta_j}} \not\in \Lambda\{P_k^{*}\}, j\neq k, 1 \le j , k \le n$. Hence, we have
$$\delta_{\bf 0} \circ P_k^{*}(D)G_i=\delta_{\bf 0} \circ P_k^{*}(D)({\bf X}^{\boldsymbol{\alpha_i}}-\sum_{j=1}^n\left(\frac{(\boldsymbol{\alpha_i})!}{(\boldsymbol{\beta_j})!}\hat{P}_j^{*}(\boldsymbol{\alpha_i})\right) {\bf X}^{\boldsymbol{\beta_j}})$$
$$=\delta_{\bf 0} \circ P_k^{*}(D){\bf X}^{\boldsymbol{\alpha_i}}-\delta_{\bf 0} \circ P_k^{*}(D)\left(\frac{(\boldsymbol{\alpha_i})!}{(\boldsymbol{\beta_k})!}\hat{P}_k^{*}(\boldsymbol{\alpha_i})\right) {\bf X}^{\boldsymbol{\beta_k}}$$
$$=(\boldsymbol{\alpha_i})!\hat{P}_k^{*}(\boldsymbol{\alpha_i})-\left(\frac{(\boldsymbol{\alpha_i})!}{(\boldsymbol{\beta_k})!}\hat{P}_k^{*}(\boldsymbol{\alpha_i})\right)(\boldsymbol{\beta_k})! $$
$$=0, ~~~~~1\leq k\leq n, 1\leq i\leq m.$$
Since $\{P_{k}, 1 \le k \le n \}$  can be expressed linearly by $\{P_{k}^{*}, 1 \le k \le n \}$, it follows that
$$\delta_{\bf 0} \circ P_k(D)G_i=0, ~~~~~1\leq k \leq n,~~ 1\leq i\leq m,$$
i.e.,
$$\delta_{\bf 0} \circ \lambda_{G\cup Q}(e^{\boldsymbol{\theta_k} {D}})G_i=0, ~~~~~1\leq k \leq n, ~~1\leq i\leq m.$$
Notice that $(\Lambda\{e^{\boldsymbol{\theta_k} {\bf X}}\}-(G\cup Q)) \cap \Lambda\{G_i\}=\emptyset$, it is easy to see
$$\delta_{\bf 0} \circ e^{\boldsymbol{\theta_k} {D}}G_i=0, ~~~~~1\leq k \leq n,~~ 1\leq i\leq m.$$
By (\ref{eqn:nonzeroCon}), we have
$$\delta_{\bf \boldsymbol{\theta_k}} G_i=0, ~~~~~1\leq k \leq n, ~~1\leq i\leq m.$$
So $G_i, 1\leq i\leq m$ vanish at $\boldsymbol{\theta_k}, 1\leq k \leq n$. It follows that $G_i, 1\leq i\leq m$ lies in $I(\Delta)$, the proof is completed.
\end{proof}

Line 30 in Algorithm \ref{alg:Groebnerbasis} shows that the ``reverse" reduced team provides all the information needed to construct the reduced Gr\"{o}bner basis. Since we use the skill in MMM algorithm, Algorithm \ref{alg:Groebnerbasis} also has a polynomial time complexity.

\begin{example}\label{example:1}
(Lagrange interpolation)
 Given the monomial order $\mathrm{grlex}(y \prec x)$, consider the  bivariate Lagrange interpolation with the interpolation conditions
		\begin{displaymath}
			   \begin{aligned}
          \Delta= {\rm span}\{\delta_{(0, 0)}, \delta_{(1, 2)}, \delta_{(2, 1)}\}.
					\end{aligned}
		\end{displaymath}			

\end{example}

		
		\begin{figure}[htbp!] \setlength{\abovecaptionskip}{0pt} \setlength{\belowcaptionskip}{0pt}		
		 \centering
		 \subfigure[Interpolation points]{
   \includegraphics[width=0.35\textwidth]{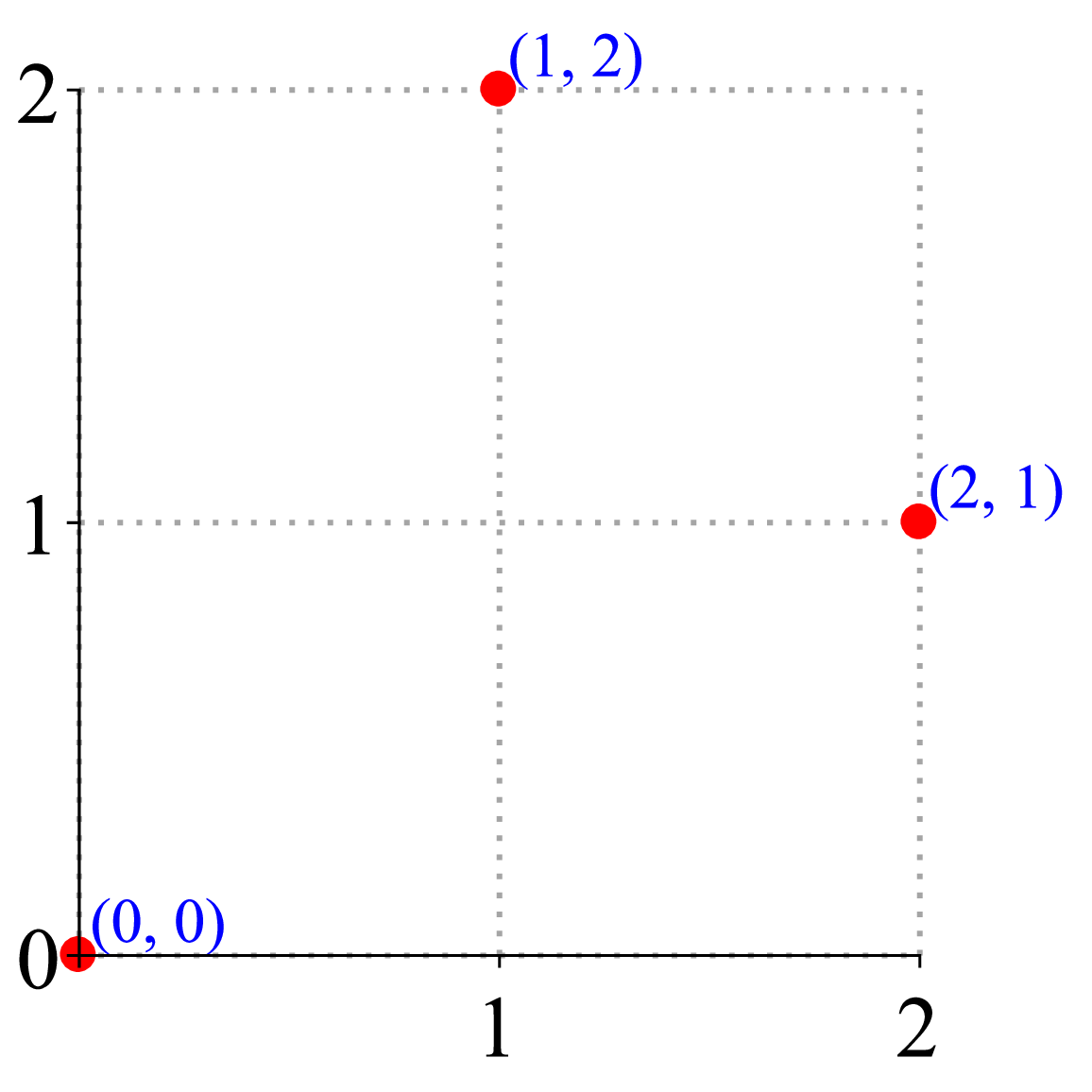}
		  \label{fig:subfig:Lag1}
		 }
\hspace{0.1\textwidth}
		  \subfigure[Quotient ring basis]{
		  \includegraphics[width=0.35\textwidth]{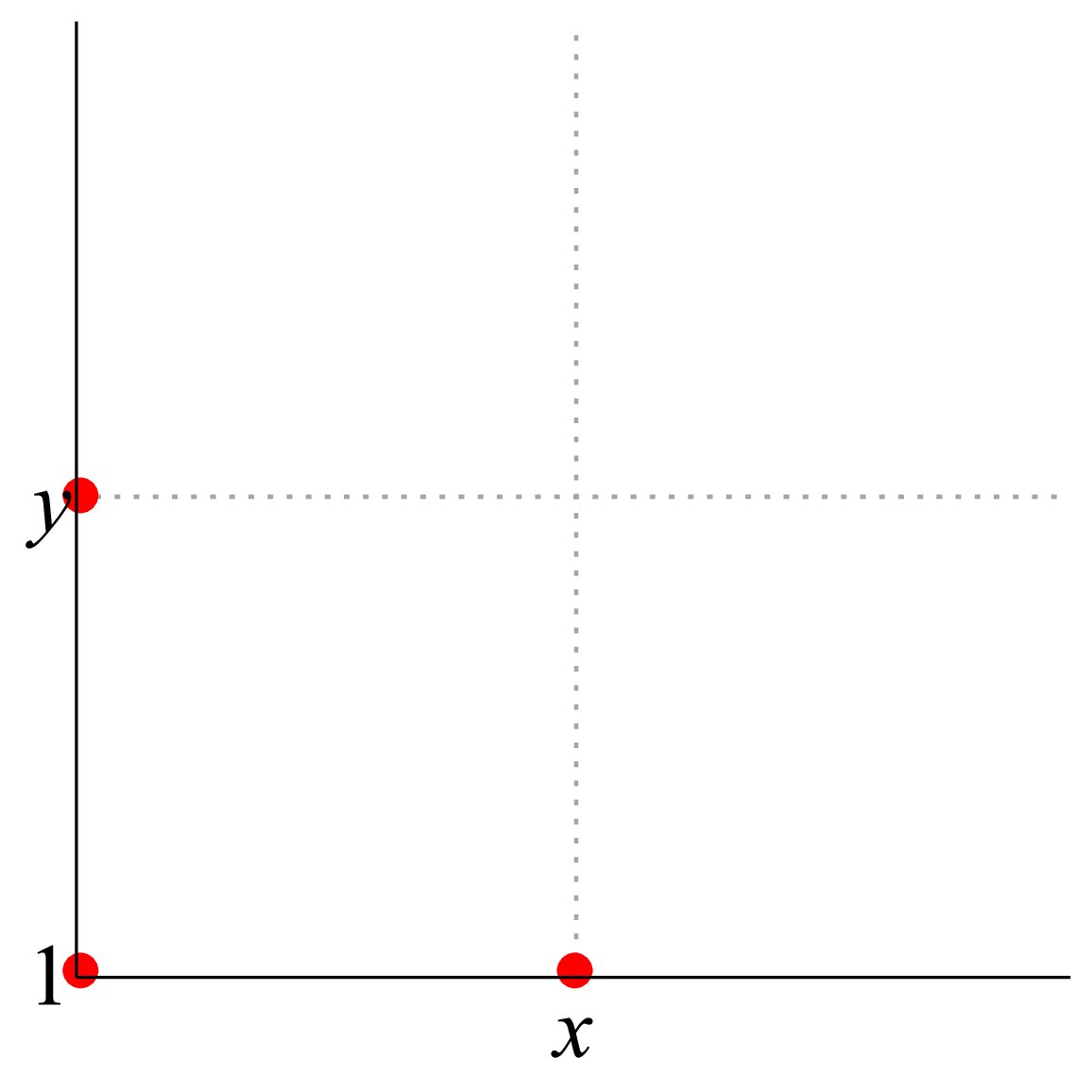}
		  \label{fig:subfig::Lag1}
		 }
		 \caption{Lagrange interpolation}
		 \label{fig:lag}
		\end{figure}
	

 By Algorithm \ref{alg:Groebnerbasis}, we get
   $$Q=\{1,y,x\},$$
   $$G=\{y^2,xy,x^2\},$$
			\begin{displaymath}
				\begin{aligned}
				    \{P_1, P_2, P_3\}  & = \{\lambda_{G\cup Q}(e^{(0, 0){\bf X}}),  \lambda_{G\cup Q}(e^{(1, 2){\bf X}}),  \lambda_{G\cup Q}(e^{(2, 1){\bf X}}) \}\\
				       & = \{1, ~ \frac{1}{2!}(x^2+4xy+4y^2) +(x+2y)+1, ~\frac{1}{2!}(4x^2+4xy+y^2) +(2x+y)+1\}.
				\end{aligned}
			\end{displaymath}		
We compute a ``reverse" reduced team of $\{P_1, P_2, P_3\}$, we get
\begin{displaymath}
				\begin{aligned}
				    \{P_{1}^{*}, P_{2}^{*}, P_{3}^{*}\}& = \{1,~(-\frac{1}{3}x^2+\frac{2}{3}xy+\frac{7}{6}y^2) +y, ~
					   (\frac{7}{6}x^2+\frac{2}{3}xy-\frac{1}{3}y^2)+x\}.
				\end{aligned}
			\end{displaymath}
%
%
%
Finally, the reduced Gr\"{o}bner basis is read from the ``reverse" reduced team $\{P_{1}^{*}, P_{2}^{*},
 P_{3}^{*}\}$ by Line 30 in Algorithm \ref{alg:Groebnerbasis}, we get
			\begin{displaymath}
			    \begin{aligned} \abovedisplayskip=2pt \belowdisplayskip=2pt
			        G_{1}& = y^2-\frac{7}{3}y+\frac{2}{3}x,\\
                    G_{2}& = xy-\frac{2}{3}y-\frac{2}{3}x, \\
			        G_{3}& = x^2+\frac{2}{3}y-\frac{7}{3}x. \\
			    \end{aligned}
			\end{displaymath}		
$\{G_1, G_2, G_3\}=\{y^2+\frac{2}{3}x-\frac{7}{3}y, xy-\frac{2}{3}x-\frac{2}{3}y, x^2-\frac{7}{3}x+\frac{2}{3}y\}$ is the reduced Gr\"{o}bner basis for $I(\Delta)$ w.r.t. $\prec$.

For general ideal interpolation, we can convert it to single point ideal interpolation. For example, given ideal interpolation conditions
		    $$\Delta=
			\left\{
				\begin{array}{l}
				    \delta_{\boldsymbol{\theta}_1} \circ {\rm span}\{P_{11}(D), P_{12}(D), \dots, P_{1s_1}(D)\}, \\
				    \delta_{\boldsymbol{\theta}_2} \circ {\rm span}\{P_{21}(D), P_{22}(D), \dots, P_{2s_2}(D)\}, \\
					\hspace*{2.5cm} \vdots\\
					\delta_{\boldsymbol{\theta}_k} \circ {\rm span}\{P_{k1}(D), P_{k2}(D), \dots, P_{ks_k}(D)\}, \\
				\end{array}
			\right.$$
where distinct points $\boldsymbol{\theta_i} \in \mathbb{F}^d, i=1,2,\dots ,k$ and $s_1+s_2+\dots +s_k=n.$ By(\ref{eqn:nonzeroCons2zeroCons}), we have
 $$\Delta=\delta_{\boldsymbol{0}} \circ {\rm span}\{P_{1}(D), P_{2}(D), \dots, P_{n}(D)\},$$
where $P_1=e^{\boldsymbol{\theta_1} {\bf X}}P_{11}, P_2=e^{\boldsymbol{\theta_1} {\bf X}}P_{12}, \dots, P_{s_1}=e^{\boldsymbol{\theta_1} {\bf X}}P_{1s_1}, P_{s_1+1}=e^{\boldsymbol{\theta_2} {\bf X}}P_{21},\dots, \\P_n=e^{\boldsymbol{\theta_k} {\bf X}}P_{ks_k}.$

The main cost of Algorithm \ref{alg:Groebnerbasis} is to calculate the rank of the matrix for obtaining the monomial basis of the quotient ring. In the case of single point ideal interpolation, if the polynomials in the interpolation conditions is a ``reverse" reduced team, then we can obtain the monomial basis of the quotient ring without calculation by Theorem \ref{thm:6}. Therefore, we get a faster algorithm for computing the reduced Gr\"{o}bner basis. We have the following algorithm for single point ideal interpolation.

\begin{algorithm}
	\caption{The reduced Gr\"{o}bner basis (Single point ideal interpolation)}
	\label{alg:Groebnerbasis-H}
	\begin{algorithmic}[1]
	\STATE{{\bf Input:} A monomial order $\prec$.}
	\STATE{\hspace*{1.1cm}The interpolation conditions
		    $$\Delta=\delta_{\boldsymbol{\theta}} \circ {\rm span}\{P_{1}(D), P_{2}(D), \dots, P_{n}(D)\}.$$}
    \STATE{\hspace*{1.1cm}where $\boldsymbol{\theta} \in \mathbb{F}^d,$ $\{P_{1}, P_{2}, \dots, P_{n}\}$ is a ``reverse" reduced team.}
	\STATE{{\bf Output:} $\{G_1,G_2,\dots, G_m\}$, the reduced Gr\"{o}bner basis for $I(\Delta)$.}	
	\STATE{//Initialization}

\STATE{$Q:=\{{\rm lm}(P_1), {\rm lm}(P_2), \dots, {\rm lm}(P_n)\}$, the monomial basis of the quotient ring, by Theorem \ref{thm:6};}
\STATE{${\rm{List}}:=\{x_i{\rm lm}(P_j), \forall 1\le i\le d, 1\le j\le n\}$;}
\STATE{${\rm{List}}:={\rm{List}}-Q$;}
\STATE{$G:=\emptyset$;}

\STATE{//Computing}

 \WHILE{${\rm{List}} \ne \emptyset$}
 \STATE{${\bf X}^{\boldsymbol{\alpha}}:= \rm{min} (\rm{List},\prec)$;}
  \STATE{$G:=G\cup \{{\bf X}^{\boldsymbol{\alpha}}\}$};
\STATE{$\rm{List}:=\rm{List}-\{{\rm{multiples~ of~}}{\bf X}^{\boldsymbol{\alpha}}\}$};

  \ENDWHILE

\STATE{$G:=\{{\bf X}^{\boldsymbol{\alpha_1}}, {\bf X}^{\boldsymbol{\alpha_2}}, \dots, {\bf X}^{\boldsymbol{\alpha_m}}\}$, the set of the leading monomials of the reduced Gr\"{o}bner basis};

 \STATE{$P_j:=\lambda_{G\cup Q}(e^{\boldsymbol{\theta} {\bf X}}P_j),~1\leq j\leq n$;}

\STATE{$\{P_j^{*}, 1\leq j\leq n\}$, a ``reverse" reduced team  w.r.t. $\prec$, by Algorithm \ref{alg:revredbasis};}
\STATE{${\bf X}^{\boldsymbol{\beta_j}}:={\rm lm}(P_j^{*}),~~1\le j\le n;$}

\FOR{$i=1:m$}
            \STATE{$G_i={\bf X}^{\boldsymbol{\alpha_i}}-\sum_{j=1}^n\left(\frac{(\boldsymbol{\alpha_i})!}{(\boldsymbol{\beta_j})!}\hat{P}_j^{*}(\boldsymbol{\alpha_i})\right) {\bf X}^{\boldsymbol{\beta_j}}$;}
	\ENDFOR
	\RETURN $\{G_1,G_2,\dots, G_m\}$.
	\end{algorithmic}
	\end{algorithm}

\section{A special example for ideal interpolation} \label{sec:6}

For general ideal interpolation, we convert it to single point ideal interpolation by (\ref{eqn:nonzeroCons2zeroCons}) and get the ``reverse" reduced team by Algorithm \ref{alg:revredbasis}, then we compute the reduced Gr\"{o}bner basis by Algorithm \ref{alg:Groebnerbasis-H}. The amount of these calculation is almost the same as MMM algorithm. However, in some special cases, we can first compute the Gr\"{o}bner basis of a single point ideal interpolation by Algorithm \ref{alg:Groebnerbasis-H}(this requires little calculation),then construct the original Gr\"{o}bner basis. The scale of the problem decreases in this case, thus the computational efficiency has been improved. A relevant example is given below.

\begin{example}
(Ideal interpolation)
Given the monomial order $\mathrm{lex}(y \prec x)$ and ideal interpolation conditions
			\begin{displaymath}
			    \Delta= \left\{
					\begin{array}{l} \abovedisplayskip=2pt \belowdisplayskip=2pt
					    \delta_{(0, 0)} \circ {\rm span}\{1, D_x, \frac{1}{2}D_x^2+D_y\}, \\
					    \delta_{(1, 2)} \circ {\rm span}\{1, D_x\}.						
					\end{array}
				\right.
			\end{displaymath}
\end{example}

	    \begin{figure}[htbp!] \setlength{\abovecaptionskip}{0pt} \setlength{\belowcaptionskip}{0pt}		
		 \centering
		 \subfigure[Interpolation points]{
		  \includegraphics[width=0.3\textwidth]{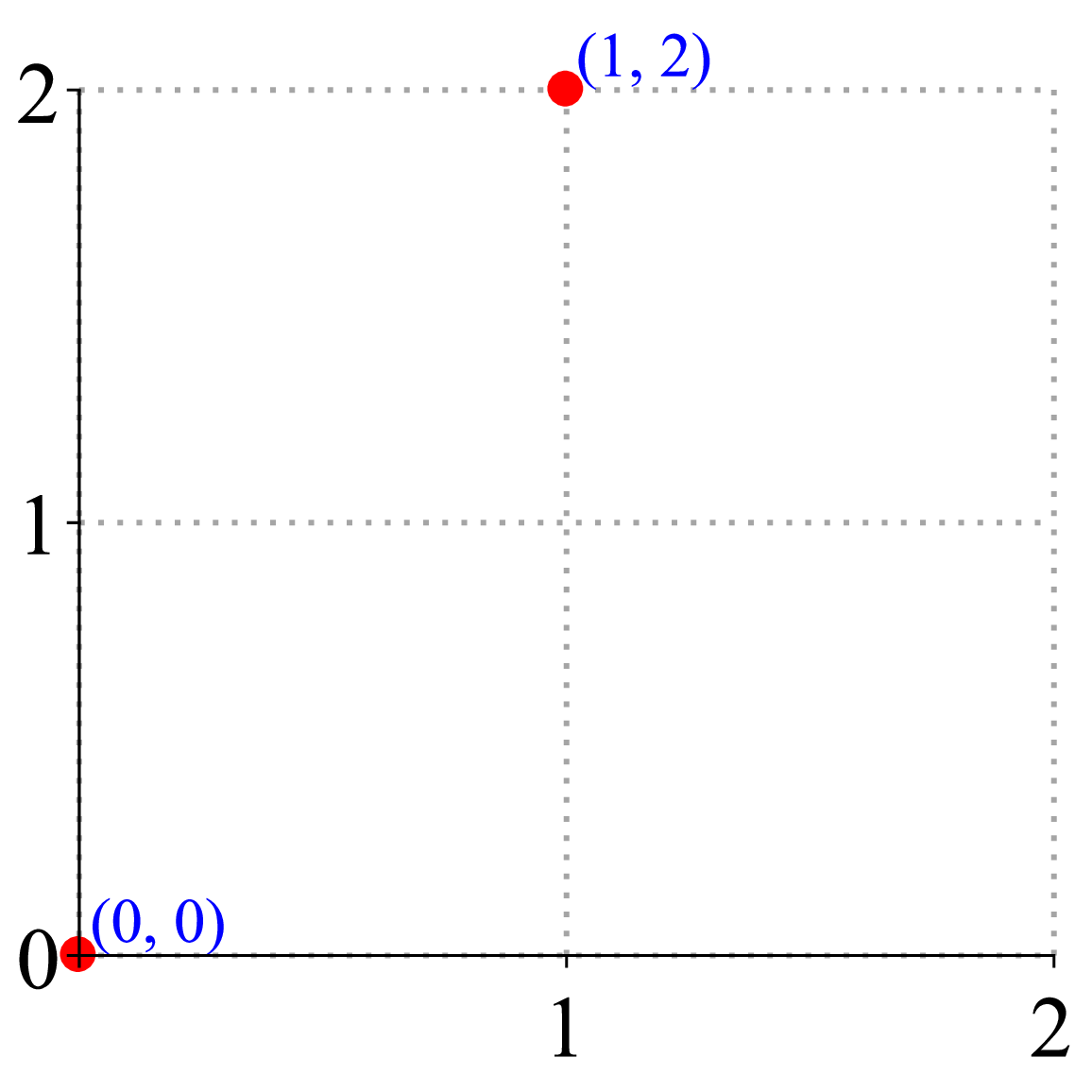}
		  \label{fig:subfig:Her1}
		 }
		 \hspace{0.1\textwidth}
		  \subfigure[Quotient ring basis]{
		  \includegraphics[width=0.3\textwidth]{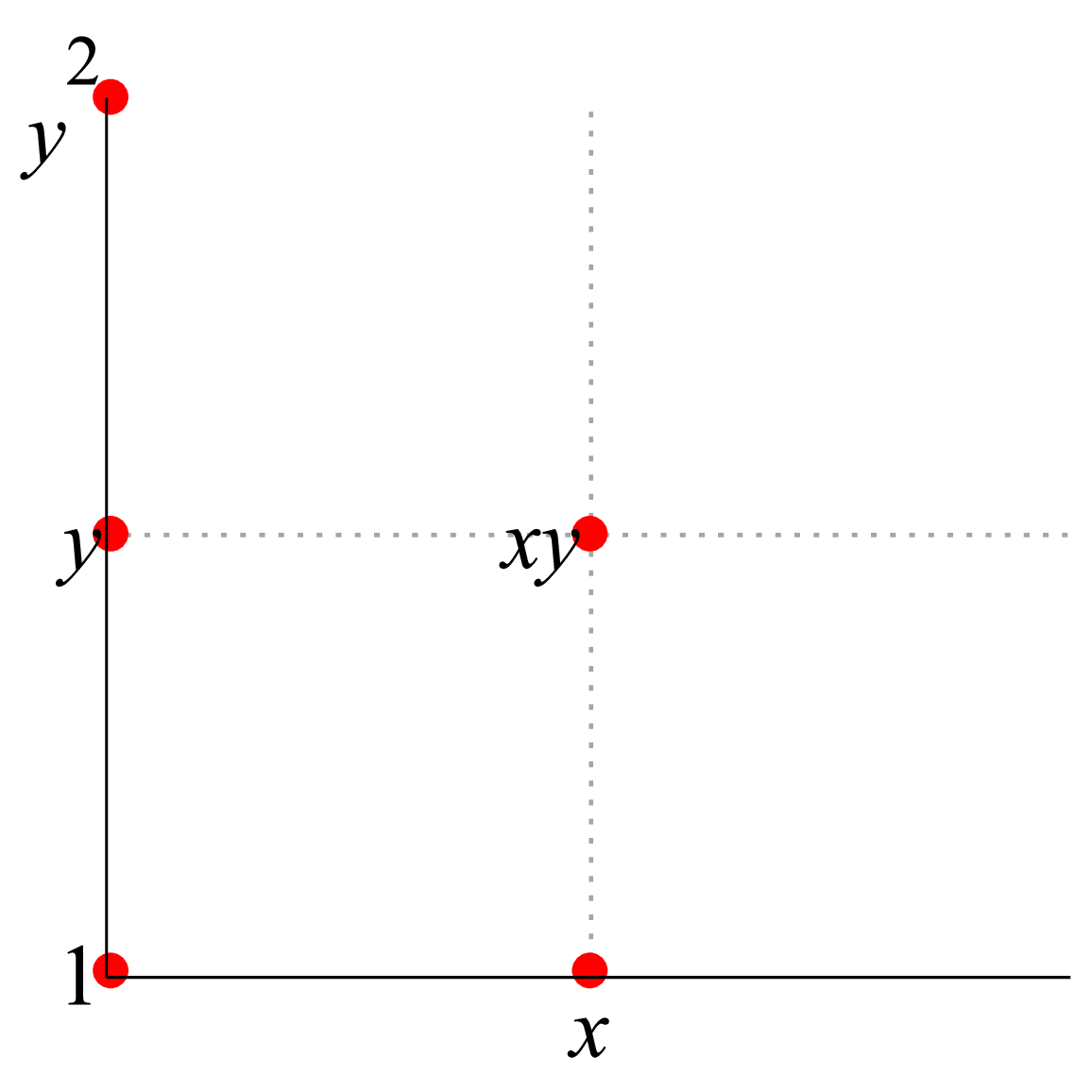}
		  \label{fig:subfig:Her2}
		 }
		 \caption{Ideal interpolation}
		 \label{fig:subfig:Her}
		\end{figure}
	

We split the original problem into two subproblems $\Delta_A$ and $\Delta_B$.

$$\Delta_A:=\delta_{(0, 0)} \circ {\rm span}\{1, D_x, \frac{1}{2}D_x^2+D_y\}.$$

Since $\{1,x,\frac{1}{2}x^2+y\}$ is a ``reverse" reduced team, $Q_A:=\{1,x,y\}$ is the monomial basis of the quotient ring $\mathbb{F}[{\bf X}]/ I(\Delta_A),$ $G_A:=\{y^2,xy,x^2\}$
is the set of the leading monomials of the reduced Gr\"{o}bner basis for $I(\Delta_A)$. By Line 21 in Algorithm \ref{alg:Groebnerbasis-H}, we get
			\begin{displaymath}
			    \begin{aligned} \abovedisplayskip=2pt \belowdisplayskip=2pt
			        G_{A1}& := y^2,\\
                    G_{A2}& := xy, \\
			        G_{A3}& := x^2-y. \\
			    \end{aligned}
			\end{displaymath}		
$\{G_{A1}, G_{A2}, G_{A3}\}=\{ y^2, xy, x^2-y\}$ is the reduced Gr\"{o}bner basis for $I(\Delta_A)$.

$$\Delta_B:=\delta_{(1, 2)} \circ {\rm span}\{1, D_x\}.$$

With the same procedure, we can get
%
			\begin{displaymath}
			    \begin{aligned} \abovedisplayskip=2pt \belowdisplayskip=2pt
			        G_{B1}&: = y-2,\\
                    G_{B2}& := x^2-2x+1. \\
			    \end{aligned}
			\end{displaymath}		
$\{G_{B1}, G_{B2}\}=\{ y-2, x^2-2x+1\}$ is the reduced Gr\"{o}bner basis for $I(\Delta_B)$.

Notice that polynomials $G_{A1}=y^2$ and $G_{B1}=y-2$ are coprime, there exist polynomials $u=\frac{1}{4}$ ,$v=-\frac{1}{4}(y+2)$ such that
$$uG_{A1}+vG_{B1}=1.$$
Let \begin{displaymath}
			    \begin{aligned} \abovedisplayskip=2pt \belowdisplayskip=2pt
			        G_{1}& := G_{A1}G_{B1}=y^3-2y^2,\\
                    G_{2}& := G_{A2}G_{B1}=xy^2-2xy,\\
           {\tilde G}_{3}& := (uG_{A1})G_{B2}+(vG_{B1})G_{A3}=x^2-\frac{1}{2}xy^2+\frac{1}{4}y^3+\frac{1}{4}y^2-y. \\
			    \end{aligned}
			\end{displaymath}	
Since $G_1,G_2,{\tilde G}_3$ all lie in $I(\Delta_A \cup \Delta_B)=I(\Delta)$, the linear combination ${\tilde G}_3-cG_1-eG_2$ lies in $I(\Delta)$, where $c=\frac{1}{4}$ is the coefficient of $y^3$ in ${\tilde G}_3$ and $e=-\frac{1}{2}$ is the coefficient of $xy^2$ in ${\tilde G}_3$.
Let $$G_{3}:={\tilde G}_3-cG_1-eG_2=x^2-xy+\frac{3}{4}y^2-y. $$
Notice that the leading term of $G_i$ in particular divides none of the nonleading terms of $G_j$, for $i,j \in \{1,2,3\}$, meanwhile the dimension of $\mathbb{F}[{\bf X}]/\langle G_1,G_2,G_3\rangle$ is 5. Therefore, $\{G_1,G_2,G_3\}$ is the reduced Gr\"{o}bner basis for $I(\Delta)$.

~\\


${\bf References}$

\bibliographystyle{elsarticle-num}

\end{document}